\documentclass[reqno]{amsart}
\usepackage{amsmath, amssymb, amsthm, epsfig}
\usepackage{hyperref, latexsym}

\usepackage{url}
\usepackage[mathscr]{euscript}

\usepackage{color}
\usepackage{fullpage} 
\usepackage{setspace}

\newcommand{\intav}[1]{\mathchoice {\mathop{\vrule width 6pt height 3 pt depth  -2.5pt
\kern -8pt \intop}\nolimits_{\kern -6pt#1}} {\mathop{\vrule width
5pt height 3  pt depth -2.6pt \kern -6pt \intop}\nolimits_{#1}}
{\mathop{\vrule width 5pt height 3 pt depth -2.6pt \kern -6pt
\intop}\nolimits_{#1}} {\mathop{\vrule width 5pt height 3 pt depth
-2.6pt \kern -6pt \intop}\nolimits_{#1}}}

\def\today{\ifcase\month\or
  January\or February\or March\or April\or May\or June\or=
  July\or August\or September\or October\or November\or December\fi
  \space\number\day, \number\year}

 \newtheorem{theorem}{Theorem}
  \newtheorem{conjecture}{Conjecture}
 \newtheorem{lemma}[theorem]{Lemma}
 \newtheorem{proposition}[theorem]{Proposition}
 
 \theoremstyle{definition}

 \theoremstyle{remark}
 \newtheorem{remark}[theorem]{Remark}

 \newcommand{\R}{\mathbb{R}}

 \newcommand{\Z}{\mathbb{Z}}

\newcommand{\wt}{\widetilde}

\newcommand{\var}{{\rm Var\,}}

\begin{document}

\title[Sharp inequalities for maximal operators on finite graphs,2]{Sharp inequalities for maximal operators on finite graphs, II}
\author[Gonz\'{a}lez-Riquelme and Madrid]{Cristian Gonz\'{a}lez-Riquelme and Jos\'e Madrid}
\date{\today}
\subjclass[2010]{26A45, 42B25, 39A12, 46E35, 46E39, 05C12.}
\keywords{Maximal operators; finite graphs; p-bounded variation; sharp constants.}

\address{IMPA - Instituto de Matem\'{a}tica Pura e Aplicada\\
Rio de Janeiro - RJ, Brazil, 22460-320.}
\email{cristian@impa.br}

\address{Department of  Mathematics,  University  of  California,  Los  Angeles (UCLA),  Portola Plaza 520, Los  Angeles,
California, 90095, USA}
\email{jmadrid@math.ucla.edu}

\allowdisplaybreaks
\numberwithin{equation}{section}

\maketitle
\begin{abstract}
Let $M_{G}$ be the centered Hardy-Littlewood maximal operator on a finite graph $G$. We find $\underset{p\to \infty}{\lim}\|M_{G}\|_{p}^{p
}$ when $G$ is the start graph ($S_n$) and the complete graph ($K_n$), and we fully describe $\|M_{S_n}\|_{p}$ and the corresponding extremizers for $p\in (1,2)$. We prove that $\underset{p\to \infty}{\lim}\|M_{S_n}\|_{p}^{p
}=\frac{1+\sqrt{n}}{2}$ when $n\ge 25$. Also, we compute the best constant ${\bf C}_{S_n,2}$ such that for every $f:V\to \mathbb{R}$ we have $\var_{2}M_{S_n}f\le {\bf C}_{S_n,2}\var_{2}f$. We prove that ${\bf C}_{S_n,2}=\frac{(n^2-n-1)^{1/2}}{n}$ for all $n\geq 3$ and characterize the extremizers. Moreover, when $M$ is the Hardy-Littlewood maximal operator on $\mathbb{Z}$, we compute the best constant ${\bf C}_{p}$ such that $\var_{p}Mf\le {\bf C}_{p}\|f\|_{p}$ for $p\in (\frac{1}{2},1)$ and we describe the extremizers.  
\end{abstract}

\section{Introduction}
Maximal operators are classical objects in analysis. They have applications in several areas of mathematics 
and have attracted interest since the beginning of the past century. 
In this manuscript we are interest in these operators acting on graphs. Given a locally finite connected graph $G=(V,E)$, endowed with the metric $d_{G}$ induced by the edges, and $f:V\to \mathbb{R}$, we define
$$M_{G}f(v)=\underset{r\ge 0}{\sup}\ \intav{B(v,r)}|f|,$$
where $B(v,r)=\{v_1\in V;d_{G}(v_1,v)\le r\}$ and $\intav{B(v,r)}|f|=\frac{\sum_{u\in B(v,r)}|f(u)|}{\mu(B(v,r))}$. Here $\mu$ is the counting measure. 
\subsection{The $p$-norm of maximal functions on finite graphs}
The $p-$norm (quasi-norm in the range $0<p<1$) of these operators is defined as
$$\|M_{G}\|_{p}:=\underset{\underset{f\neq 0}{f:V\to \mathbb{R}}}{\sup}\frac{\|M_{G}f\|_{p}}{\|f\|_{p}},$$
where $\|g\|_{p}=\left(\displaystyle\sum_{v\in V}|g(v)|^{p}\right)^{\frac{1}{p}},$ for any $g:V\to \mathbb{R}.$

The study of the $p-$norm of maximal operators acting on finite graphs was initiated by Soria and Tradacete in \cite{SORIA}. They established optimal bounds for $\|M_G\|_{p}$ in the range $0<p\leq 1$, and showed that the complete graph $K_n$ and the star graph $S_n$ are the
extremal graphs attaining, respectively, the lower and upper bounds. Since then, the study of maximal operators acting on graphs, specially on the complete $M_{K_n}$ and star graphs $M_{S_n}$ has attracted the attention of many authors. One of the main goals of this manuscript is to try to extend this notions to the range $p>1.$ 

The case $p>1$ presents several difficulties. For instance, in the case $p\le 1$ the concavity of the function $x\mapsto x^p$ was used in \cite[Lemma 2.5]{SORIA} in order to prove that $\|M_{G}\|_p$ is attained by taking some Dirac's delta. That reduction simplifies matters in that case and it is not available for $p>1$. In fact, the known extemizers for $p>1$ are quite different (see Section 2 and 3).

The first progress in this direction was achieved in the authors previous work \cite{GRM}. There, we obtained the precise values of the $2-$norm of $M_{K_n}$ and $M_{S_n}$. Moreover, we found extremizers in these cases. The techniques used in such results do not work for general $p>1$. In Section 2 we address this problem, we fully characterize the extremizers for $\|M_{S_n}\|_p$ for all $p\in (1,2]$. Moreover, we obtain some partial characterization for this objects for all $p>2$, and we obtain a similar result for the extremizers of $\|M_{K_n}\|_p$ in the range $p>1$.\\

For a given $p>1$ and $G$, it could be difficult to determine the value of $\|M_{G}\|_{p}.$ That happens even in the model cases $G=K_n$ and $G=S_n$. However, it was proved by Soria and Tradacete (see Proposition 3.4 in \cite{SORIA}) that 
$$
\left(1+\frac{n-1}{2^p}\right)\leq \|M_{S_n}\|^p_{p}\leq \left(\frac{n+5}{2}\right).
$$
They also presented similar bounds for $\|M_{K_n}\|_{p}^{p}$. We notice that both lower bounds go to $1$ when $p\to \infty.$

In Section 3 we discuss the behavior of both $\|M_{S_n}\|_{p}^{p}$ and $\|M_{K_n}\|_{p}^{p}$. In particular, we prove that $$\inf_{p>0} \|M_{G_n}\|_p^{p}>1,$$ for any graph $G$ of $n$ vertices. This improves qualitatively the aforementioned estimates of Soria and Tradacete.     
Also, in Section 3 we prove that for all $n\geq 25$ we have
$$
\lim_{p\to\infty}\|M_{S_n}\|^p_{p}=\frac{1+\sqrt{n}}{2}.
$$
Moreover, we obtain a similar result for $M_{K_n}$.\\
\subsection{The p-variation of maximal functions}
Given a locally finite connected simple graph $G=(V,E)$ and $p>0$ we define, for every $g:V\to \mathbb{R},$
$$\var_{p}g=\left(\underset{\underset{(v_1,v_2)=e}{e\in E}}{\sum}|g(v_1)-g(v_2)|^{p}\right)^{\frac{1}{p}},$$
where we write $(v_1,v_2)=e$ if the edge $e$ joins $v_1$ and $v_2$.
As in the definition of $\|M_{G}\|_{p}$, it is natural to define $${\bf C}_{G,p}:=\sup_{f:V\to \mathbb{R};\var_{p}f>0}\frac{\var M_{G}f}{\var_{p}f}.$$
Motivated by a large family of regularity results for maximal operators acting on Sobolev spaces (see \cite{CarneiroSurvey}), in \cite{Liu2020} Liu and Xue computed ${\bf C}_{K_n}$ and ${\bf C_{S_n}}$ in the case $n\le 3$. Also, they made some conjectures for these values for $n>3$. These conjectures were largely establish by the authors in \cite{GRM}. There, the authors proved that
$$
{\bf C}_{K_n,p}=1-\frac{1}{n}\ \ \text{for all}\ p\geq \frac{\log 4}{\log 6} \ \ \text{and}\ \  {\bf C}_{S_n,p}=1-\frac{1}{n} \ \ \text{for all}\ \ p\in[1/2,1].
$$
Moreover, in all the previous situations the extremizers are delta functions.
In the case $p>1$, delta functions are not extremizers for the $p-$variation of $M_{S_n}$. In Section 4 we find the precise value of ${{\bf C}_{S_n,2}}$. Moreover, we fully describe the extremizers in this case.\\

The study of the regularity properties of the discrete Hardy-Littlewood maximal operators acting on real valued functions defined on the integers was initiated by Bober, Carneiro, Hughes and Pierce in \cite{BCHP2012}, some other interesting results were later obtained in \cite{CH2012}, \cite{CM2015} and \cite{Madrid2016}. 
In Section 5, we obtain some complementary results extending
\cite[Theorem 1]{Madrid2016} to the range $p\in [\frac{1}{2},1)$. In \cite[Theorem 1]{Madrid2016} the second author proved that 
$$\var_{1}M_{\mathbb{Z}}f\le 2\|f\|_{1},$$
when considering $\mathbb{Z}$ as a graph where consecutive numbers are joined by an edge.
This inequality is sharp. The motivation behind this inequality is to try to get an intuition about which is the optimal constant $C$ in the estimate
$$\var_{1}M_{\mathbb{Z}}f\le C\var_{1}f,$$
that was proved to be true for $C=(2\cdot 120\cdot 2^{12}\cdot 300+4)$ in \cite{temur2017regularity}. Since $2\|f\|_{1}\ge \var_{1}f$ it is believed that $C=1$ is the optimal constant, but this remains an open problem. In Section 5 we find the best constant $C_p$ such that 
$$\var_{p}M_{\mathbb{Z}}f\le C_{p}\|f\|_{p}$$
for $p\in [\frac{1}{2},1].$
This motivates us to make some conjectures. We also establish the analogous optimal result for $p=\infty$.


\section{Extremizers for the $p-$norm of maximal operators on graphs}
In this section we prove the existence of extremizers for the $p$-norm and provide some further properties about these functions.
\begin{proposition}\label{prop: existen extremizers}
Let $G=(V,E)$ be a connected finite graph and $p>0$. We have that there exists $f:V\to \mathbb{R}_{\ge 0}$ such that 
$$\frac{\|M_{G}f\|_{p}}{\|f\|_{p}}=\|M_{G}\|_{p}$$
\end{proposition}
\begin{proof}
We write $|V|=n$ and $V=:\{a_1,\dots,a_n\}$. Given $y:=(y_1,\dots,y_n)\in [0,1]^{n}\cap \{\underset{{i=1,\dots,n}}{\max} y_{i}=1\}=:A$ we define $f_y:V\to \mathbb{R}_{\ge 0}$ by $f_y(a_i)=y_i$. We observe that $M_{G}f_y(a_i)$ is continuous with respect to $y$ in $A$ (since is the maximum of continuous functions). Then, the function $\frac{\|M_{G}f_{y}\|_{p}}{\|f_{y}\|_{p}}$ is continuous with respect to $y$ in $A$. Thus it achieves its maximum at a point $y_0\in A$. We claim that $$\frac{\|M_{G}f_{y_0}\|}{\|f_{y_0}\|_{p}}=\|M_{G}\|_{p}.$$ In fact, for every $g:V\to \mathbb{R}_{\ge 0}$ we have that the quantity $$\frac{\|M_{G}g\|_{p}}{\|g\|_{p}}$$ remains unchanged by applying the transformation $$g\mapsto \frac{g}{\max_{i=1,\dots,n}g(a_i)}.$$
This last function is equal to $f_y$ for some $y\in A$, from where we conclude.  

\end{proof}
Our next results intend to characterize 
the extremizers when $G=K_n$ and $G=S_n$.
\begin{proposition}\label{twovaluescomplete}
Let $K_n=(V,E)$ be the complete graph with $n$ vertices where $V=\{a_1,a_2\dots,a_n\}$ and let $p>1$. If $$\frac{\|M_{K_n}f\|_{p}}{\|f\|_{p}}=\|M_{K_n}\|_{p},$$ then $f$ only takes two values. 
\end{proposition}
\begin{proof}
First, by taking a Dirac's delta is easy to see that $\|M_{K_n}\|_{p}>1.$ Now, assume that $f$ attains the supremum. We have then that $$\sum_{i=1}^{r}f(a_i)^p+(n-r)m^p={\|M_{K_n}\|}^p\left(\sum_{i=1}^n f(a_i)^p\right).$$ Therefore, by H\"older's inequality we have that \begin{align*}
(n-r)m^p&=\left(\|M_{K_n}\|_{p}^p-1\right)\left(\sum_{i=1}^r f(a_i)^p\right)+\|M_{K_n}\|_{p}^p\left(\sum_{i=r+1}^n f(a_i)^p\right)\\
&\ge r(\|M_{K_n}\|_{p}^p-1)\left(\frac{\sum_{i=1}^r f(a_i)}{r}\right)^p+
(n-r)\|M_{K_n}\|_{p}^p\left(\frac{\sum_{i=r+1}^n f(a_i)}{n-r}\right)^p. 
\end{align*}
Then, if we take the function $\widetilde{f}(a_i)=\frac{\sum_{i=1}^r f(a_i)}{r}$ for $i=1,\dots,r,$ and $\widetilde{f}(a_i)=\frac{\sum_{i=r+1}^n f(a_i)}{n-r}$ for $i=r+1,\dots,n,$ we have $$\frac{\|M_{K_n}\widetilde{f}\|_{p}}{\|\widetilde{f}\|_p}\ge \|M_{K_n}\|_{p},$$ with equality if and only if $f(a_i)=\widetilde{f}(a_i)$ for every $i=1,\dots n.$ So, we conclude.
\end{proof}
We also get the following result.
\begin{proposition}\label{threevaluesstar}
Let $S_n=(V,E)$ be the star graph with $n$ vertices $V=\{a_1,a_2,\dots,a_n\}$ with center at $a_1$ and let $p\geq 1$. There exists $f:V\to \mathbb{R}$ with $$\frac{\|M_{S_n}f\|_{p}}{\|f\|_{p}}=\|M_{S_n}\|_{p},$$ such that  $f(a_1)=\max{f}$ and $f_{|V\setminus{a_1}}$ takes (at most) two values.
\end{proposition}
\begin{proof}
By Proposition \ref{prop: existen extremizers} there exists $g\ge 0$ such that $$\frac{\|M_{S_n}g\|_{p}}{\|g\|_{p}}=\|M_{S_n}\|_{p}.$$
Now, we proceed in three steps.
\\
{\it{Step 1: We can assume that $g(a_1)\geq g(a_j)$ for all $j\in\{2,3,\dots,n\}$}}.
We assume without loss of generality that $g(a_2)\ge \dots \ge g(a_r)\ge g(a_1)\ge \dots \ge g(a_n),$ consider $\widetilde{g}(x):=g(x)$ for $x\in V\setminus\{a_2,a_1\}$, $\widetilde{g}(a_2):=g(a_1)$ and $\widetilde{g}(a_1):=g(a_2).$ We observe that $$M_{S_n}\widetilde{g}(a_1)^p+M_{S_n}\widetilde{g}(a_2)^p=g(a_2)+\max\left\{m^p,\left(\frac{g(a_2)+g(a_1)}{2}\right)^p\right \}\ge M_{S_n}g(a_2)^p+M_{S_n}g(a_1)^p.$$ Also, for $x\in V\setminus\{a_1,a_2\},$ we have that $M_{S_n}\widetilde{g}(x)\ge M_{S_n}g(x)$ since $$\max \left\{m,\frac{g(x)+\widetilde{g}(a_1)}{2},g(x)\right\}\ge \max \left\{m,\frac{g(x)+g(a_1)}{2},g(x)\right\}.$$ Therefore, we get $$\sum_{i=1}^{n}M_{S_n}\widetilde{g}(a_1)^p\ge \sum_{i=1}^{n}M_{S_n}g(a_i)^p,$$ since clearly we have that $\|\widetilde g\|_p=\|g\|_p$. We conclude that $$\frac{\|M_{S_n}\widetilde g\|_{p}}{\|\widetilde g\|_{p}}=\|M_{S_n}\|_{p}.$$  So, we can assume that $g(a_1)\ge g(a_j)$ for every $j.$\\

Then, we assume without loss of generality that  $g(a_1)\ge \dots \ge g(a_r)\ge 2m-g(a_1)
>g(a_{r+1})\ge \dots g(a_n)$. 

{\it{Step 2: We can assume that $g(a_{r+1})=g(a_{r+2})=\dots=g(a_n)$.}}
We consider the function $\widetilde g:V\to\R$ defined by  $\widetilde{g}(a_i)=\frac{\sum_{i=r+1}^n g(a_i)}{n-r}$ for every $i=r+1\dots n$ and $\widetilde{g}=g$ otherwise. We have (similarly as in the previous proposition) that $$\frac{\|M_{S_n}\widetilde{g}\|_{p}}{\|\widetilde{g}\|_{p}}\ge \|M_{S_n}\|_{p}.$$ Therefore, we can assume that $g(a_i)=g(a_n)$ for every $i\ge r+1.$\\

{\it{Step 3: We can assume that $g(a_2)=g(a_3)=\dots=g(a_r)$.}}
Now consider $$\widetilde{g}(a_i)=\left(\frac{\sum_{j=2}^r g(a_j)^p}{r-1}\right)^{\frac{1}{p}}$$ for $i=2,\dots,r$ and $\widetilde{g}=g$ elsewhere. Since $\sum_{i=1}^n|\widetilde{g}(a_i)|^p=\sum_{i=1}^{n}|g(a_i)|^p$ is enough to prove that $$\sum_{i=1}^n|M_{S_n}\widetilde{g}(a_i)|^p\ge \sum_{i=1}^n|M_{S_n}g(a_i)|^p.$$ Let us observe first that $$\widetilde{m}:=\frac{\sum_{i=1}^n\widetilde{g}(a_i)}{n}\ge \frac{\sum_{i=1}^n g(a_i)}{n}=m,$$since $$(r-1)\left(\frac{\sum_{i=2}^r g(a_i)^p}{r-1}\right)^{1/p}\ge \sum_{i=2}^r g(a_i) $$ by H\"older's inequality. Thus, 
for $i=r+1,\dots,n$ we have that $M_{S_n}\widetilde{g}(a_i)\ge \widetilde{m}\ge m=M_{S_n}g(a_i).$ 
Also, we observe that for all $i\in\{2,\dots,r\}$ we have $$M_{S_n}\widetilde{g}(a_i)\ge \frac{g(a_1)+\widetilde{g}(a_i)}{2}.$$  
So, it is enough to prove that $$(r-1)\left(\frac{g(a_1)+\left(\frac{\sum_{j=2}^r g(a_j)^p}{r-1}\right)^{1/p}}{2}\right)^p=\sum_{i=2}^{r}\left(\frac{g(a_1)+\widetilde{g}(a_i)}{2}\right)^p\ge \sum_{i=2}^{r}M_{S_n}g(a_i)^p=\sum_{i=2}^{r}\left(\frac{g(a_1)+g(a_i)}{2}\right)^p,$$ but that is equivalent to $$g(a_1)+\left(\frac{\sum_{j=2}^r g(a_j)^p}{r-1}\right)^{1/p}\ge \left(\frac{\sum_{i=2}^{r}(g(a_1)+g(a_i))^p}{r-1}\right)^{1/p},$$ which is a consequence of Minkowsky's inequality. From where we conclude our required result. 

\end{proof}



\begin{theorem}
For all $n\geq 3$, let $S_n=(V,E)$ be the star graph with $n$ vertices $V=\{a_1,a_2,\dots,a_n\}$ with center at $a_1$. For all $p\in(1,2]$ we have that $$\|M_{S_n}\|_p=\left(\sup_{x\in [0,1)}\frac{1+(n-1)(\frac{x+1}{2})^p}{1+(n-1)x^p}\right)^{\frac{1}{p}}.$$
\end{theorem}
\begin{proof}First, let us assume that $n>3$. Let $f:V\to\R$ be a function such that $\frac{\|M_{S_n}f\|_{p}}{\|f\|_{p}}=\|M_{S_n}\|_{p}$ as in Proposition \ref{threevaluesstar}. After a normalization (if necessary) we can assume that $f(a_1)=1$. By Proposition \ref{threevaluesstar} we have that $f_{|V\setminus{a_1}}$ only takes two values, let us say $x\le y\le 1$, $x$ s-times and $y$ t-times. We will prove that $x=y$. We observe that if both $x$ and $y$ satisfy $x,y\ge 2m_{f}-1$ by the same argument as in Proposition \ref{threevaluesstar} we conclude that $x=y.$ The same happens if $x,y\le 2m_{f}-1.$ So, the only case remaining is when $x< 2m_{f}-1<y$. 
Then, we observe that by taking a Dirac's delta in $a_1$ we have that $\|M_{S_n}\|_{p}^p\ge 1+\frac{n-1}{2^p}\ge \frac{n+3}{4}.$ Let us first assume that $y<1,$ given $\varepsilon$ such that $$1>y+\varepsilon>2\left(\frac{1+t(y+\varepsilon)+sx}{n}\right)-1,$$ we consider $f_{\varepsilon}:V\to \mathbb{R}$ defined by $f_{\varepsilon}(a_i)=f(a_i)+\varepsilon$ for all $a_i$ such that $f(a_i)=y$ and $f_{\varepsilon}=f$ elsewhere. If we consider the function  (defined in a neighborhood of $0$) $$L(\varepsilon):=\|M_{S_n}f_{\varepsilon}\|_{p}^p-\|M_{S_n}\|_{p}^p\|f_{\varepsilon}\|_{p}^p,$$ we have $L(0)=0$ and $L(\varepsilon)\le 0$ in a neighborhood of $0$. Therefore $L'(0)=0$, that is
\begin{align}\label{derivative}
0&=\left(1+t\left(\frac{1+y+\varepsilon}{2}\right)^p+s\left(\frac{1+sx+t(y+\varepsilon)}{n}\right)^p-\|M_{S_n}\|_{p}^p(1+t(y+\varepsilon)^p+sx^p)\right)'\nonumber\\
&=\frac{tp(\frac{1+y}{2})^{p-1}}{2}+s\frac{tp}{n}\left(\frac{1+sx+ty}{n}\right)^{p-1}-tp\|M_{S_n}\|_{p}^p(y^{p-1}).
\end{align}
We observe that in fact $y\ge \frac{1+x}{2},$ if not we would have $m_f\leq \frac{1+x}{2},$ a contradiction. However, that implies that $y\ge m_f$ since is equivalent to $(s+1)y=(n-t)y\ge sx+1,$ which is true because $(s+1)\left(\frac{1+x}{2}\right)\ge sx+1$. Then  $$\frac{n+3}{4}y^{p-1}\le \|M_{S_n}\|_{p}^{p}y^{p-1}\le\frac{(\frac{1+y}{2})^{p-1}}{2}+\frac{n-2}{n}\left(\frac{1+sx+ty}{n}\right)^{p-1}<\frac{(\frac{1+y}{2})^{p-1}}{2}+\frac{n-2}{n}y^{p-1}.$$ Then \begin{align}\label{desigualdadcr}
\frac{n-1}{2}+\frac{4}{n}<\left(\frac{1+y}{2y}\right)^{p-1}\le \frac{1+y}{2y}.
\end{align}
Also, we observe that since $2m_f-1>x>0,$ we have $m_f>\frac{1}{2},$ so if $x<y\le \frac{1}{4},$ we have $$\frac{1}{2}<m_f=\frac{1+sx+ty}{n}<\frac{1+\frac{s+t}{4}}{n}=\frac{1+\frac{n-1}{4}}{n}.$$ Therefore, $\frac{n}{4}<\frac{3}{4},$ a contradiction. So $y>\frac{1}{4},$ then $$\frac{1+y}{2y}<\frac{10}{4}=\frac{5}{2}.$$ Then, by \eqref{desigualdadcr}, we obtain
$\frac{n-1}{2}+\frac{4}{n}<\frac{5}{2},$ that is false for $n\ge 4.$ We conclude this case. The only remaining case is when $y=1.$ In this case, we have that (where $L$ is defined in an interval $(\delta,0]$, with $\delta$ close to 0) $$\frac{L(0)-L(-\varepsilon)}{\varepsilon}\ge 0.$$ Therefore taking $-\varepsilon\to 0^{-}$, similarly as we obtained \eqref{derivative}, we have that $0\le \frac{tp}{2}+s\frac{tp}{n}m^{p-1}-tp\|M_{S_n}\|_{p}^p$ and that implies $\frac{n+3}{4}\le \frac{1}{2}+1,$ which is false for $n>3.$ Therefore we conclude this case. The remaining case $n=3$ is treated as follows. By the same argument (and notation) above, if $x<2m_f-1<y$, we have $\|M_{S_3}\|_{p}^p\ge 1+\frac{2}{2^p}\ge \frac{3}{2}$ and $y\ge \frac{1}{2}.$ Proceeding as before, similarly as we obtained \eqref{derivative}, we have that $$\|M_{S_3}\|_{p}^py^{p-1}\le \frac{(\frac{1+y}{2})^{p-1}}{2}+\frac{1}{3}(m_f)^{p-1}\le \frac{(\frac{1+y}{2})^{p-1}}{2}+\frac{1}{3}y^{p-1},$$ 
therefore $\frac{7}{3}\le (\frac{1+y}{2y})^{p-1}\le (\frac{3}{2})^{p-1}\le \frac{3}{2},$ a contradiction. So, we conclude.
\end{proof}
\begin{remark}
An adaptation of the proof above also shows that for any $p>1$ there exists a positive constant $N(p)$ such that for any $n>N(p)$ we have $${\bf C}_{S_n,p}=\left(\sup_{x\in [0,1)}\frac{1+(n-1)(\frac{x+1}{2})^p}{1+(n-1)x^p}\right)^{\frac{1}{p}}.$$
\end{remark}

\section{Asymptotic behavior of $\|M_G\|_p$}
In the next propositions we study the behavior of $\|M_{K_n}\|_{p}$ and $\|M_{S_n}\|_{p}$ as $p\to \infty.$
We start with an useful elementary lemma.
\begin{lemma}\label{cgrlemalhopital}
Assume that for $\{p_{k}\}_{k\in \mathbb{N}}\subset [1,\infty)$ such that $p_k\to \infty$  we have $x_{1,p_k},\dots, x_{n,p_k}\ge 0$ such that $\underset{k\to \infty}{\lim} x_{i,p_k}^{p_k}\to x_{i}<\infty$, for every $i=1,\dots,n$. Then we have that 
\begin{align*}
\lim_{k\to \infty}\left(\frac{\sum_{i=1}^{n}x_{i,p_k}}{n}\right)^{p_k}=(x_{1}x_{2}\dots x_n)^{\frac{1}{n}}.    
\end{align*}
\end{lemma}
\begin{proof}
By AM-GM inequality we have $$\left(\frac{\sum_{i=1}^{n}x_{i,p_k}}{n}\right)^{p_k}\ge \left(x_{1,p_k}^{p_k}x_{2,p_k}^{p_k}\dots x_{n,p_k}^{p_k}\right)^{\frac{1}{n}}\to (x_{1}x_{2}\dots x_n)^{\frac{1}{n}}.$$ So, we just need to prove that $$\underset{k\to \infty}{\lim \sup} \left(\frac{\sum_{i=1}^{n}x_{i,p_k}}{n}\right)^{p_k}\le (x_{1}x_{2}\dots x_n)^{\frac{1}{n}}.$$ Given $\varepsilon>0$, for $k$ big enough we have that $x_{i,p_k}\le (x_i+\varepsilon)^{\frac{1}{p_k}}$ for every $i=1,\dots, n$. Then, we observe that  $$\lim_{k\to \infty}\left(\frac{\sum_{i=1}^{n}(x_i+\varepsilon)^{\frac{1}{p_k}}}{n}\right)^{p_k}=((x_1+\varepsilon)(x_2+\varepsilon)\dots (x_n+\varepsilon))^{\frac{1}{n}}$$ by the L'Hospital rule after applying $\log$ in both sides. Therefore, for every given $\varepsilon>0$ we have $$\underset{k\to \infty}{\lim \sup} \left(\frac{\sum_{i=1}^n x_{i,p_i}}{n}\right)^{p_k}\le ((x_1+\varepsilon)(x_2+\varepsilon)\dots (x_n+\varepsilon))^{\frac{1}{n}},$$ from where we conclude.  
\end{proof}
Now we continue by analyzing the behavior of $\|M_{K_n}\|_{p}^{p}$ when $p$ goes to $\infty$.  
In the following lemma we construct an example that helps us to achieve that goal. 

\begin{lemma}\label{A.B K_n lim inf }
Let $K_n=(V,E)$ be the complete graph with $n$ vertices\ $V=\{a_1,a_2\dots,a_n\}$. Then, $$\underset{p\to \infty}{\lim \inf}\ \|M_{K_n}\|_{p}^{p}\ge \underset{\alpha>1, k\in \{1,\dots,n\}}{\sup}\frac{k\alpha^{\frac{n}{k}}+\alpha(n-k)}{k\alpha^{\frac{n}{k}}+n-k}.$$ 
\end{lemma}
\begin{proof}[Proof of Lemma \ref{A.B K_n lim inf }]
For fixed $k$ and $\alpha>1$ we define the function $f:V\to \mathbb{R}_{\ge 0}$ fiven by $f_{p}(a_i)=\frac{n\alpha^{\frac{1}{p}}-(n-k)}{k}$ for $i\le k,$ and $f_{p}(a_i)=1$ elsewhere. Thus we have $m_{p}:=\frac{\sum_{i=1}^{n}f_{p}(a_i)}{n}=\alpha^{\frac{1}{p}}.$ Moreover, we observe that $$\lim_{p\to \infty} \left(\frac{n\alpha^{\frac{1}{p}}-(n-k)}{k}\right)^{p}=\alpha^{\frac{n}{k}},$$ therefore
\begin{align*}
\underset{p\to \infty}{\lim \inf}\ \|M_{K_n}\|_{p}^{p}&\ge \lim_{p\to \infty} \frac{k\left(\frac{n\alpha^{\frac{1}{p}}-(n-k)}{k}\right)^{p}+(n-k)m_{p}^{p}}{k\left(\frac{n\alpha^{\frac{1}{p}}-(n-k)}{k}\right)^{p}+(n-k)}\\
&=\frac{k\alpha^{\frac{n}{k}}+(n-k)\alpha}{k\alpha^{\frac{n}{k}}+(n-k)},
\end{align*}
from where we conclude.
\end{proof}

We observe that the previous proof gives us the lower bound $$\|M_{K_n}\|_{p}^{p}\ge \underset{\alpha>1, k\in \{1,\dots,n\}}{\sup} \frac{k\left(\frac{n\alpha^{\frac{1}{p}}-(n-k)}{k}\right)^{p}+(n-k)\alpha}{k\left(\frac{n\alpha^{\frac{1}{p}}-(n-k)}{k}\right)^{p}+(n-k)},$$
for every $p\ge 1.$
Now we claim that this lower bound gives essentially the behavior when $p\to \infty$ for $\|M_{K_n}\|_{p}^p.$ This is the content of the following theorem. 
\begin{theorem}\label{A.B K_n thm}
Let $n\ge 3$ and let $K_n=(V,E)$ be the complete graph with $n$ vertices\ $V=\{a_1,a_2\dots,a_n\}$.
Then, $$\lim_{p\to \infty}\|M_{K_n}\|_{p}^p=\underset{\alpha>1, k\in \{1,\dots,n\}}{\sup}\frac{k\alpha^{\frac{n}{k}}+\alpha(n-k)}{k\alpha^{\frac{n}{k}}+n-k}.$$
\end{theorem}
\begin{proof}[Proof of Theorem \ref{A.B K_n thm}]
 By the previous lemma we just need to prove that $$\lim \sup_{p\to \infty} \|M_{K_n}\|_{p}^p\le \underset{\alpha>1, k\in \{1,\dots,n\}}{\sup}\frac{k\alpha^{\frac{n}{k}}+\alpha(n-k)}{k\alpha^{\frac{n}{k}}+n-k}:=C_n.$$ Observe that $C_n>1 $ since $\alpha>1$. Moreover, by Proposition \ref{prop: existen extremizers} for all $p>1$ there exists a function $f_p:V\to\R$ such that $\|M_{K_n}\|_p=\frac{\|M_{K_n}f_p\|_p}{\|f_p\|_p}$. Let us assume that there exists a sequence $p_i\to \infty,$ such that: 
\begin{equation}\label{thm 7 seq argument C_n}
\frac{\|M_{K_n}f_{p_i}\|_{p_{i}}^{p_{i}}}{\|f_{p_i}\|_{p_i}^{p_i}}>c,
\end{equation}
for a fixed constant $c>C_n$.
We assume without lose of generality that $f(a_1)\ge f(a_2)\dots \ge f(a_n)$.
By Proposition \ref{twovaluescomplete}, we know that $f_{p_i}$ only takes two values, if the minimum of this two values is $0$, after a normalization (if necessary) we could assume $f_{p_i}(a_j)=1$ for $j\le k_0<n$ and $f_{p_i}=0$ elsewhere, then 
\begin{align*}
\frac{\|M_{K_n}f_{p_i}\|_{p_i}^{p_i}}{\|f_{p_i}\|_{p_i}^{p_i}}=\frac{k_0+(n-k_0)(\frac{k_0}{n})^{p_i}}{k_0}\le 1+(n-1)\left(\frac{n-1}{n}\right)^{p_i}\to 1,    
\end{align*}
a contradiction for $p_i$ big enough. So we can assume without loss of generality  that $f_{p_i}$ takes two different positive values, and after a normalization, we can assume that the minimum value of $f_{p_i}$ is $1$. Let us call the other value by $y_{p_i}>1$. Let us take a subsequence of $p_i$ (that we also call $p_i$) such that $f_{p_i}(a_r)=y_{p_i}$ for $r\le k$ (for some fixed $k\in \{1,\dots,n\}$) and $f_{p_i}(a_r)=1$ elsewhere. We claim that $y_{p_i}\to 1.$ In fact, if there exist a subsequence (that we also call $p_i$) such that $y_{p_{i}}\ge \rho>,1$ we have $$\frac{m_{p_i}}{y_{p_i}}=\frac{k+(n-k)\frac{1}{y_{p_i}}}{n}\le \frac{k+(n-k)\frac{1}{\rho}}{n}<1.$$ Therefore \begin{align*}
\frac{\|M_{K_n}f_{p_i}\|_{p_i}^{p_i}}{\|f_{p_i}\|_{p_i}^{p_i}}=\frac{ky_{p_i}^{p_i}+(n-k)m_{p_i}^{p_i}}{ky_{p_i}^{p_i}+(n-k)}&\le 1+(n-k)\left(\frac{m_{p_i}}{y_{p_i}}\right)^{p_i}\\
&\le 1+(n-k)\left(\frac{k+(n-k)\frac{1}{\rho}}{n}\right)^{p_i}\to 1,      
\end{align*}
a contradiction. Now we claim that the $y_{p_i}^{p_i}$ are uniformly bounded. Assume that for a subsequence (that we also call $y_{p_i}$) we have $y_{p_i}^{p_i}\to \infty$. We consider the function $g(x)=nx^{\frac{n-\frac{1}{2}}{n}}-kx-(n-k)=0,$ we observe that $g(x)\ge 0$ for $x\in \left[1,\left(\frac{(n-\frac{1}{2})}{k}\right)^{2n}\right].$ In fact $g(1)=0$ and $g$ is increasing in  $\left[1,\left(\frac{\left(n-\frac{1}{2}\right)}{k}\right)^{2n}\right]$ since $g'(x)=\left(n-\frac{1}{2}\right)x^{\frac{-1}{2n}}-k\geq 0.$ Now, for $p_i$ big enough we have $y_{p_i}\in \left[1,\left(\frac{\left(n-\frac{1}{2}\right)}{k}\right)^{2n}\right]$. Thus $ny_{p_i}^{\frac{n-\frac{1}{2}}{n}}-ky_{p_i}-(n-k)\ge 0$ and then $$m_{p_i}=\frac{ky_{p_i}+n-k}{n}\le y_{p_i}^{\frac{n-\frac{1}{2}}{n}}.$$ Therefore \begin{align*}
\frac{\|M_{K_n}f_{p_i}\|_{p_i}^{p_i}}{\|f_{p_i}\|_{p_i}^{p_i}}=\frac{ky_{p_i}^{p_i}+(n-k)m_{p_i}^{p_i}}{ky_{p_i}^{p_i}+(n-k)}&\le 1+(n-k)\left(\frac{m_{p_i}}{y_{p_i}}\right)^{p_i}\\
&\le 1+(n-k)\left(y_{p_i}^{-\frac{1}{2n}}\right)^{p_i}\to 1, 
\end{align*}   
reaching a contradiction. So, we have that $y_{p_i}^{p_i}$ are uniformly bounded. Let us take a subsequence of $p_i$ (that we also denote $p_i$ for simplicity) such that $y_{p_i}^{p_i}$ and $m_{p_i}^{p_i}$ converges. Let us write $\underset{p_i\to \infty}{\lim} y_{p_i}^{p_i}=\alpha_1$ and $\underset{p_i\to \infty}{\lim} m_{p_k}^{p_k}=\alpha_2$. Then, by Lemma \ref{cgrlemalhopital} we have (taking $x_{s,p_k}=y_{p_k}$ for $s\le k$ and $x_{s,p_k}=1$ for $s>1$) $\alpha_2=\alpha_1^{\frac{k}{n}}.$

This implies
\begin{align*}
\lim_{p_i\to \infty}\frac{\|M_{K_n}f_{p_i}\|_{p_i}^{p_i}}{\|f_{p_i}\|_{p_i}^{p_i}}=\lim_{p_i\to \infty}\frac{ky_{p_i}^{p_i}+(n-k)m_{p_i}^{p_i}}{ky_{p_i}^{p_i}+(n-k)}=\frac{k\alpha_2^{\frac{n}{k}}+(n-k)\alpha_2}{k\alpha_{2}^{\frac{n}{k}}+(n-k)}\leq C_n.
\end{align*}
Then, it is not possible to have a sequence like in \eqref{thm 7 seq argument C_n}, therefore
$$\underset{p\to \infty}{\lim \sup}\ \|M_{K_n}\|_{p}^p\le C_n$$
as desired.
\end{proof}
Now we start analyzing the behavior of $\|M_{S_n}\|_{p}^p$ when $p$ goes to $\infty$. In the following lemma we construct an example that helps us to achieve this goal. 
\begin{lemma}\label{lemma assymp S_n lower bound}
Let $n\ge 3$ and let $S_n=(V,E)$ be the star graph with $n$ vertices $V=\{a_1,a_2,\dots,a_n\}$ with center at $a_1$. Then, $$\underset{p\to \infty}{\lim \inf}\ \|M_{S_n}\|_{p}^{p}\ge \frac{1+\sqrt{n}}{2}.$$
\end{lemma}
\begin{proof} 
For fixed $k>1$ we define $y_{k,p}=2k^{\frac{1}{p}}-1,$ we observe that $\left(\frac{1+y_{k,p}}{2}\right)^p=k.$ Let us consider the function $f_{k,p}:V\to \mathbb{R}_{\ge 0}$ by $f_{k,p}(a_1)=y_{k,p}$ and $f_{k,p}(a_i)=1$ for $i>1$. Then, we have
\begin{align*}
\|M_{S_n}\|_{p}^p\ge \left(\frac{\|M_{S_n}f_{k,p}\|_{p}}{\|f_{k,p}\|_{p}}\right)^p=\frac{\left(2k^{\frac{1}{p}}-1\right)^p+(n-1)k}{\left(2k^{\frac{1}{p}}-1\right)^p+(n-1)}. 
\end{align*}
We observe that by L'Hospital $\underset{p\to \infty}{\lim}\left(2k^{\frac{1}{p}}-1\right)^p=k^2,$ therefore we have 
$$\underset{p\to \infty}{\lim \inf}\ \|M_{S_n}\|_{p}^{p}\ge \frac{k^2+(n-1)k}{k^2+(n-1)}.$$
By taking $k=\sqrt{n}+1$ we have $\frac{k^2+(n-1)k}{k^2+(n-1)}=\frac{\sqrt{n}+1}{2},$
from where we conclude our proposition.
\end{proof}
We observe that the proof above gives us the estimate $$\|M_{S_n}\|_{p}^{p}\ge \frac{\left(2(1+\sqrt{n})^{\frac{1}{p}}-1\right)^{p}+(n-1)(1+\sqrt{n})}{\left(2(1+\sqrt{n})^{\frac{1}{p}}-1\right)^{p}+(n-1)}$$ for every $p\ge 1.$ Moreover, we observe that $\left(2(1+\sqrt{n})^{\frac{1}{p}}-1\right)^{p}$ is an increasing function on $p$. This is the case because the derivative of $p\log\left(2(1+\sqrt{n})^{\frac{1}{p}}-1\right)$ is $$\log\left(2(1+\sqrt{n})^{\frac{1}{p}}-1\right)-\frac{(1+\sqrt{n})^{\frac{1}{p}}\log(1+\sqrt{n})}{(2(1+\sqrt{n})^{\frac{1}{p}}-1)p}\ge \log(2(1+\sqrt{n})^{\frac{1}{p}}-1)-\frac{\log(1+\sqrt{n})}{p}\ge 0.$$ Thus, we have that
$$\frac{\left(2(1+\sqrt{n})^{\frac{1}{p}}-1\right)^{p}+(n-1)(1+\sqrt{n})}{\left(2(1+\sqrt{n})^{\frac{1}{p}}-1\right)^{p}+(n-1)}$$ is decreasing with respect to $p$. Then $$\|M_{S_n}\|_{p}^{p}\ge \lim_{t\to \infty} \frac{\left(2(1+\sqrt{n})^{\frac{1}{t}}-1\right)^{t}+(n-1)(1+\sqrt{n})}{\left(2(1+\sqrt{n})^{\frac{1}{t}}-1\right)^{t}+(n-1)}=\frac{\sqrt{n}+1}{2},$$ for all $p\ge 1.$
Note that this lower bound is better than the one observed by Soria and Tradacete  $1+\frac{n-1}{2^p}\le\|M_{S_n}\|_{p}^p$ (see Proposition 3.4 in \cite{SORIA})  whenever $p>\frac{\log(\sqrt{n}+1)}{\log(2)}+1.$ \\

Let us define $$\|M_{S_{n}}\|_{p}^{*}:=\sup_{y\ge 1}\left(\frac{y^p+(n-1)(\frac{1+y}{2})^p}{y^p+(n-1)}\right)^{\frac{1}{p}}.$$ Our next goal is to analyze the relation between this object and $\|M_{S_n}\|_p$. We start observing that by definition $\|M_{S_{n}}\|_{p}^{*}\le \|M_{S_n}\|_{p}$. Also, we have the following.

\begin{lemma}\label{star^{*}}
Let $n\ge 3$. The following identity holds $$\lim_{p\to \infty} \left(\|M_{S_{n}}\|_{p}^{*}\right)^{p}=\frac{1+\sqrt{n}}{2}.$$
\end{lemma}
\begin{proof}
We start observing that the proof of Lemma \ref{lemma assymp S_n lower bound} also works for $\|M_{S_n}\|^*_p$. Then, it is enough to prove that $$\underset{p\to \infty}{\lim \sup}\ (\|M_{S_n}\|_{p}^{*})^{p}\le \frac{1+\sqrt{n}}{2}.$$ Let us assume that there exists a sequence $p_k\to \infty$ and $y_{p_k}>1$ such $$\frac{y_{p_k}^{p_k}+(n-1)\left(\frac{1+y_{p_k}}{2}\right)^{p_k}}{y_{p_k}^{p_k}+(n-1)}\ge c>\frac{1+\sqrt{n}}{2}.$$
We observe that $y_{p_k}\to 1$. In fact if there exists a subsequence $k_j$ such that $y_{p_{k_j}}\ge \rho>1$, we have $\frac{1+y_{p_{k_j}}}{2y_{p_{k_j}}}\le \frac{1}{2\rho}+\frac{1}{2}<1,$ then 
\begin{align*}
\frac{y_{p_{k_j}}^{p_{k_j}}+(n-1)\left(\frac{1+y_{p_{k_j}}}{2}\right)^{p_{k_j}}}{y_{p_{k_j}}^{p_{k_j}}+(n-1)}\le 1+(n-1)\left(\frac{1}{2\rho}+\frac{1}{2}\right)^{p_{k_j}}\to 1. 
\end{align*}
Therefore, this cannot be the case. 
Now we prove that the $y_{p_k}^{p_k}$ are uniformly bounded. In fact, since $y_{p_k}<2$ for $k$ big enough, we have $y_{p_k}^{\frac{3}{4}}\ge \frac{y_{p_k}+1}{2},$ since $2x^{\frac{3}{4}}-x-1\ge 0$ for $x\in [1,2]$. Therefore, if $y_{p_{k_j}}^{p_{k_j}}\to \infty$ we have $$\frac{y_{p_{k_j}}^{p_{k_j}}+(n-1)(\frac{1+y_{p_{k_j}}}{2})^{p_{k_j}}}{y_{p_{k_j}}^{p_{k_j}}+(n-1)}\le \frac{y_{p_{k_j}}^{p_{k_j}}+(n-1)(y_{p_{k_j}})^{\frac{3p_{k_j}}{4}}}{y_{p_{k_j}}^{p_{k_j}}+(n-1)}\to 1.$$
So, we have that $y_{p_k}^{p_k}$ are uniformly bounded. Let us take a subsequence of $\{p_k\}_{k\in \mathbb{N}}$ (that we also denote $\{p_k\}_{k\in \mathbb{N}}$) such that $y_{p_k}^{p_k}$ and $\left(\frac{1+y_{p_k}}{2}\right)^{p_k}$ converges. Let us write $\underset{k\to \infty}{\lim} y_{p_k}^{p_k}=\alpha_1$ and $\underset{k\to \infty}{\lim} \left(\frac{1+y_{p_k}}{2}\right)^{p_k}=\alpha_2$. By Lemma \ref{cgrlemalhopital} (with $n=2$, $x_{1,p_k}=y_{p_k}$ and $x_{2,p_k}=1$) we have $\alpha_2=\sqrt{\alpha_1}$. 
Then, observe that 
\begin{align*}
\lim_{k\to \infty}\frac{y_{p_k}^{p_k}+(n-1)(\frac{1+y_{p_k}}{2})^{p_k}}{y_{p_k}^{p_k}+(n-1)(\frac{1+y_{p_k}}{2})^{p_k}}=\frac{\alpha_2^2+(n-1)\alpha_2}{\alpha_2^2+(n-1)}\le \frac{\sqrt{n}+1}{2},    
\end{align*}
since this last inequality is equivalent to $\alpha_2^2(\sqrt{n}-1)-2(n-1)\alpha_2+(n-1)(\sqrt{n}+1)\ge 0,$ and this is true because $\alpha_2^2(\sqrt{n}-1)-2(n-1)\alpha_2+(n-1)(\sqrt{n}+1)=(\sqrt{n}-1)(\alpha_2-(\sqrt{n}+1))^2.$ This conclude the proof. 
\end{proof}

We conclude this section describing the asymptotic behavior of $\|M_{S_n}\|_{p}^{p}$ as $p\to\infty$. 
\begin{theorem}
Fix $n\in \mathbb{N}$. Let $S_n=(V,E)$ be the star graph with $n$ vertices $V=\{a_1,a_2,\dots,a_n\}$ with center at $a_1$. For $n\ge 25$ we have
\begin{align*}
    \lim_{p\to \infty}\|M_{S_n}\|_{p}^{p}=\frac{1+\sqrt{n}}{2}.
\end{align*}
\end{theorem}
\begin{proof}
We choose $f_p:V\to \mathbb{R}_{\ge 0}$ such that $\frac{\|M_{S_n}f_p\|_{p}^p}{\|f_p\|_{p}^p}=\|M_{S_n}\|_{p}^p$. First we observe that for all $p$ sufficiently large we have $1+\frac{n-1}{2^p}<\frac{\sqrt{n}+1}{2}$. Then, by Proposition \ref{threevaluesstar}, we can assume that $f_{p}(a_1)=y_p\ge 1=f_{p}(a_2)=f_{p}(a_{3})=\dots =f_{p}(a_{s+1})\ge x_p=f_{p}(a_{s+2})=\dots=f_{p}(a_n)$.
Moreover, we assume that $\frac{y_p+x_p}{2}<\frac{\sum_{i=1}^nf_{p}(a_i)}{n}=:m_p$ (otherwise, we would have that $x_p\geq 2m_p-f_p(a_1)$, and we can proceed as in  the Step 3 of the proof of Proposition \ref{threevaluesstar} to conclude that $x_p=1$). Notice that the case $s=1$ is not possible since then $\frac{x_p+y_p}{2}\ge \frac{y_{p}+(n-1)x_p}{n}=m_p$. Let us assume that there exists $1<s<n-1$ and a sequence $p_k\to \infty$ such that 
\begin{align*}
   \frac{\|M_{S_n}f_{p_k}\|_{p_k}^{p_k}}{\|f_{p_k}\|_{p}^p}=\frac{y_{p_k}^{p_k}+s\left(\frac{1+y_{p_k}}{2}\right)^{p_k}+(n-s-1)m_{p_k}^{p_k}}{y_{p_k}^{p_k}+s+(n-s-1)x_{p_k}^{p_k}}> \frac{1+\sqrt{n}}{2}.
\end{align*}
First we observe that $y_{p_k}\to 1$. If not, there exists a subsequence of $\{p_k\}_{k\in \mathbb{N}}$ (that we also call $\{p_k\}_{k\in \mathbb{N}}$) such that $y_{p_k}>\rho>1$. Then $$\frac{m_{p_k}}{y_{p_k}}\le \frac{y_{p_k}+1}{2y_{p_k}}\le \frac{1}{2}+\frac{1}{2\rho}<1$$ and therefore $$\frac{y_{p_k}^{p_k}+s\left(\frac{1+y_{p_k}}{2}\right)^{p_k}+(n-s-1)m_{p_k}^{p_k}}{y_{p_k}^{p_k}+s+(n-s-1)x_{p_k}^{p_k}}\to 1,$$ a contradiction. Now we claim that the $y_{p_k}^{p_k}$ are uniformly bounded. Assume that for a subsequence (that we also call $\{y_{p_k}\}_{k\in \mathbb{N}}$) we have $y_{p_k}^{p_k}\to \infty$. First observe that for $k$ big enough we have $y_{p_k}>1$. Then, as in the proof of Lemma \ref{star^{*}}, we have that for $p_k$ big enough $$y_{p_k}^{\frac{3}{4}}\ge \frac{y_{p_k}+1}{2}\ge m_{p_k},$$ from where we have 
$$\frac{\|M_{S_n}f_{p_k}\|_{p_k}^{p_k}}{\|f_{p_k}\|_{p_k}^{p_k}}\le \frac{y_{p_k}^{p_k}+(n-1)y_{p_k}^{\frac{3p_k}{4}}}{y_{p_k}^{p_k}+s+(n-s-1)x_{p_k}^{p_k}}\to 1.$$ Reaching a contradiction. Therefore we can take a subsequence (that we also call $\{p_k\}_{k\in \mathbb{N}}$) such that $\underset{k\to \infty}{\lim} y_{p_k}^{p_k}=\alpha_1$, $\underset{k\to \infty}{\lim} \left(\frac{1+y_{p_k}}{2}\right)^{p_k}=\alpha_2$, $\underset{k\to \infty}{\lim} m_{p_k}^{p_k}=\alpha_3$ and $\underset{k\to \infty}{\lim} x_{p_k}^{p_k}=\alpha_4$. By Lemma \ref{cgrlemalhopital} we have that $\alpha_2=\sqrt{\alpha_1}$ and $\alpha_3=\alpha_1^{\frac{1}{n}}\alpha_4^{\frac{n-s-1}{n}}.$ Therefore $$\lim_{k\to \infty}\frac{\|M_{S_n}f_{p_k}\|_{p_k}^{p_k}}{\|f_{p_k}\|_{p_k}^{p_k}}\le \frac{\alpha_{2}^2+s\alpha_2+(n-s-1)\alpha_{1}^{\frac{1}{n}}\alpha_{4}^{\frac{n-s-1}{n}}}{\alpha_{2}^{2}+s+(n-s-1)\alpha_{4}},$$ we claim that this last expression is bounded above by $\frac{\sqrt{n}+1}{2}$ in our setting, from where we would conclude. 
Since $\alpha_4\le 1$ it is enough to prove that 
\begin{align}\label{finalinequalitycgr}
\alpha_2^2+s\alpha_2+(n-s-1)\alpha_2^{\frac{2}{n}}\le \left(\frac{\sqrt{n}+1}{2}\right)\left(\alpha_2^{2}+s+(n-s-1)\alpha_4\right).
\end{align}
To this end is sufficient to prove $$\alpha_2^{2}+s\alpha_2+(n-s-1)\alpha_2^{\frac{2}{n}}\le \left(\frac{\sqrt{n}+1}{2}\right)\left(\alpha_2^{2}+s\right).$$
We observe now that since $m_{p_k}\ge \frac{y_{p_k}+x_{p_k}}{2}$ we have $2(y_{p_k}+2s+2(n-s-1)x_{p_k})\ge ny_{p_k}+nx_{p_k}$ and then $2s+(n-2(s+1))x_{p_k}\ge (n-2)y_{p_k}$. Since $1,x_{p_k}<y_{p_k}$, if we assume $n-2(s+1)\ge 0$ we would get  $$(n-2)y_{p_k}\ge 2sy_{p_k}+(n-2(s+1))y_{p_k}>2s+(n-2(s+1))x_{p_k}\ge (n-2)y_{p_k},$$ a contradiction. Therefore we have 
$n-2(s+1)\le -1$ and then \begin{align}\label{condicionsobres}
n\le 2s+1.    
\end{align}
Now we assume that $n\ge 25$. We distinguish among two cases, first when $\alpha_2\ge \left(\frac{6}{5}\right)^{n}$. Here $$\alpha_2^2+s\alpha_2+(n-s-1)\alpha_2^{\frac{2}{n}}\le \alpha_2^{2}+(n-1)\alpha_2 \le 2\alpha_2^{2},$$ since $(n-1)\le (\frac{6}{5})^{n}\le \alpha_2,$ where we use that $n\ge 25$. Then since $2\alpha_2^2\le \left(\frac{\sqrt{25}+1}{2}\right)\alpha_2^2$ we conclude this case. Now we consider the case where $\alpha_2 \le \left(\frac{6}{5}\right)^{n}$. Here we have $\alpha_2^{\frac{2}{n}}\le \left(\frac{6}{5}\right)^2$. Therefore we just need to prove that 
\begin{align*}
\alpha_2^2+s\alpha_2+(n-s-1)\left(\frac{6}{5}\right)^2\le \left(\frac{\sqrt{n}+1}{2}\right)\left(\alpha_2^{2}+s\right),   
\end{align*}
or equivalently $$0\le \left(\frac{\sqrt{n}-1}{2}\right)\alpha_2^2-s\alpha_2+s\left(\frac{\sqrt{n}+1}{2}\right)-(n-s-1)\left(\frac{6}{5}\right)^2.$$ We just need to verify then that the discriminant of that equation is lesser than $0$. That is $$s^2<4\left(\frac{\sqrt{n}-1}{2}\right)\left[s\left(\frac{\sqrt{n}+1}{2}\right)-(n-s-1)\left(\frac{6}{5}\right)^2\right]=(n-1)s-2(\sqrt{n}-1)(n-s-1)\left(\frac{6}{5}\right)^2.$$
Since $(n-1)s=s^2+s(n-s-1)$ we just need $2\left(\frac{6}{5}\right)^2(\sqrt{n}-1)< \frac{n-1}{2}$ (given that $s\ge \frac{n-1}{2}$ by \eqref{condicionsobres}) or equivalently $4\left(\frac{6}{5}\right)^2-1<\sqrt{n}$. Since the left hand side is lesser than $5$ we conclude this and therefore we conclude \eqref{finalinequalitycgr}, from where the theorem follows.  
\end{proof}
\begin{remark}
From the previous proof it can be deduced that, if we define $$\mathcal{A}=\left\{(s,\alpha_{2},\alpha_4)\in \{1,\dots, n-2\}\times[1,\infty)\times[0,1]; \alpha_{2}^\frac{2}{n}\alpha_{4}^{\frac{n-s-1}{n}}>\alpha_2\sqrt{\alpha_4} \right\},$$ for $n\in [3,24]$ we have
$$\underset{p\to \infty}{\lim} \|M_{S_n}\|_{p}^{p}=\max\left\{\frac{1+\sqrt{n}}{2},\underset{(s,\alpha_2,\alpha_4)\in \mathcal{A}}{\sup}\frac{\alpha_{2}^2+s\alpha_2+(n-s-1)\alpha_{1}^{\frac{1}{n}}\alpha_{4}^{\frac{n-s-1}{n}}}{\alpha_{2}^{2}+s+(n-s-1)\alpha_{4}}\right\}.$$
However, for $n<25$, to compare the inner terms in the right hand side is more difficult.
\end{remark}

\section{The $p$-variation of maximal operators on graphs}
In order to compute ${\bf C}_{G,p}$ it is useful to study the functions that attain this supremum. Now we prove that actually these extremizers exist. 
\begin{proposition}
Given any connected simple finite graph $G=(V,E)$ and $p\in (0,\infty)$ there exists $f:V\to \mathbb{R}_{\ge 0}$ such that 
$$\frac{\var_{p}M_{G}f}{\var_{p}f}={\bf C}_{G,p}$$
\end{proposition}
\begin{proof}
We write $|V|=n$ and $V=:\{a_1,a_2,\dots,a_n\}.$ Given $y=:(y_1,\dots, y_n)\in [0,1]^{n}\cap \left\{\underset{i=1,\dots,n}{\max}y_i=1\right\}\cap\left\{\underset{i=1,\dots,n}{\min}y_i=0\right\}=:A,$ we define $f_y:V\to \mathbb{R}_{\ge 0}$ by $f_{y}(a_i)=y_i$. We observe that $M_{G}f_y(a_i)$ is continuous in such set for any $i=1,\dots, n$. Therefore $\frac{\var_{p}M_{G}f_y}{\var_{p}f_y}$ is continuous with respect to $y$ in $A$ since the denominator is never $0$. Thus it attains its maximum at a point $y_0\in A$. We claim that 
$$\frac{\var_{p}M_{G}f_{y_0}}{\var_{p}f_{y_0}}={\bf C}_{G,p}.$$
In fact, for every $g:V\to \mathbb{R}_{\ge 0}$ we have that
the value $\frac{\var_{p}M_{G}g}{\var_{p}g}$ remains unchanged by doing the transformation $$g\mapsto \frac{g-\min_{i=1,\dots,n}g(a_i)}{\max_{i=1,\dots,n} g(a_i)}.$$
This last function is equal to $f_{y}$ for some $y,$ from where we conclude. 
\end{proof}

\subsection{The $2$-variation of $M_{S_n}$}
For all $p\geq 1$, it was proved by the authors in \cite{GRM} that $\var_p M_{K_n}f\leq \left(1-\frac{1}{n}\right)\var_pf$, for any real valued function $f$ defined on the vertices of $K_n$. The equality occurs when $f$ is a  delta function. The analogous problem for the star graph $S_n$ is more challenging, it was observed by the authors in \cite{GRM} that in this case, delta functions are not extremizers. Our next result solves this problem for $p=2$.   

\begin{theorem}
Let $n\geq3$ and let $S_n=(V,E)$ be the star graph with $n$ vertices $V=\{a_1,a_2,\dots,a_n\}$ with center at $a_1$. The following inequality holds
$$
\var_2 M_{S_n}f\leq\left(\frac{[(n-1)^2+n-2]^{1/2}}{n}\right)\var_2f
$$
for all $f:V\to\R$. 
Moreover, this result is optimal.
\end{theorem}

\begin{proof} The proof of this result is divided in two cases, the case 2 is divided in many steps.

{\it{Case 1: $f(a_1)\leq m_f$.}}
In this case $M_{S_n}f(a_1)=m_f$. If $M_{S_n}f(a_1)\geq M_{S_n}f(a_i)\geq m_f$ for all $2\leq i\leq n$ then the result is trivial.

Then, we assume without loss of generality that
$M_{S_n}f(a_i)>M_{S_n}f(a_1)$ for all $i\in\{2,3,\dots,k\}$ for some $2\leq k\leq n$ and $M_{S_n}f(a_i)=M_{S_n}f(a_1)=m_f$ for all $i\in\{k+1,k+2,\dots,n\}$.

We have that
\begin{equation}\label{LHS}
    (\var_2(M_{S_n}f))^2=\sum_{i=2}^{k}(f(a_i)-m_f)^2.
\end{equation}
Moreover, for all $i\in\{2,3,\dots,k\}$ we have that
\begin{align}\label{cada termino}
    0<(f(a_i)-m_f)^{}&=\frac{1}{n^{}}\left((n-1)f(a_i)-\sum_{j=1,j\neq i}^{k}f(a_j)-\sum_{j=k+1}^{n}f(a_j)\right)^{}\nonumber\\
    &=\frac{1}{n^{}}\left((n-1)(f(a_i)-f(a_1))-\sum_{j=2,j\neq i}^{k}(f(a_j)-f(a_1))+\sum_{j=k+1}^{n}(f(a_1)-f(a_j))\right)^{}
\end{align}
Let 
$$
S^{+}:=\left\{i\in\{2,3,\dots,k\}; -\sum_{j=2,j\neq i}^{k}(f(a_j)-f(a_1))+\sum_{j=k+1}^{n}(f(a_1)-f(a_j))>0\right\},
$$
and $S^{-}:=\{2,3,\dots,k\}\setminus{S^{+}}$.
Then by \eqref{LHS} and \eqref{cada termino} we have that
\begin{align}\label{I+II}
    (\var_2M_{S_n}f)^2\leq \frac{(n-1)^2}{n^2}\sum_{i\in S^{-}}(f(a_i)-f(a_1))^2+\sum_{i\in S^+}(f(a_i)-m_f)^2,
\end{align}
and
\begin{align}\label{s+}
    \sum_{i\in S^+}(f(a_i)-m_f)^2=&\frac{(n-1)^2}{n^2}\sum_{i\in S^+}(f(a_i)-f(a_1))^2\nonumber\\
    &+\frac{2(n-1)}{n^2}\sum_{i\in S^+} (f(a_i)-f(a_1))\left(-\sum_{j=2,j\neq i}^{k}(f(a_j)-f(a_1))+\sum_{j=k+1}^{n}(f(a_1)-f(a_j))\right)\nonumber\\
    &+\frac{1}{n^2}\sum_{i\in S^+}\left(-\sum_{j=2,j\neq i}^{k}(f(a_j)-f(a_1))+\sum_{j=k+1}^{n}(f(a_1)-f(a_j))\right)^2.
\end{align}
Also, we observe that, since $f(a_1)\leq m_f$, then $$\sum_{i=2}^{k}(f(a_i)-f(a_1))\geq \sum_{i=k+1}^{n}f(a_1)-f(a_i),$$ therefore 
\begin{align}\label{ultimo termino}
    &\sum_{i\in S^+}\left(-\sum_{j=2,j\neq i}^{k}(f(a_j)-f(a_1))+\sum_{j=k+1}^{n}(f(a_1)-f(a_j))\right)^2\nonumber\\
    &\leq \left[\sum_{i\in S^+}\left(-\sum_{j=2,j\neq i}^{k}(f(a_j)-f(a_1))+\sum_{j=k+1}^{n}(f(a_1)-f(a_j))\right)\right]^2\nonumber\\
    &\leq \left(|S^+| \sum_{j=k+1}^{n}(f(a_1)-f(a_j))- (|S^+|-1) \sum_{j=2}^{k}(f(a_j)-f(a_1)) \right)^2\nonumber\\
    &\leq \left(\sum_{j=k+1}^{n}(f(a_1)-f(a_j)) \right)^2\nonumber\\
    &\leq (n-k)\sum_{j=k+1}^{n}(f(a_1)-f(a_j))^2.
\end{align}

Moreover, by the AM-GM inequality we have that
\begin{align}\label{medio termino}
 &2(n-1)\sum_{i\in S^+} (f(a_i)-f(a_1))\left(-\sum_{j=2,j\neq i}^{k}(f(a_j)-f(a_1))+\sum_{j=k+1}^{n}(f(a_1)-f(a_j))\right)\nonumber\\
 &\leq \sum_{i\in S^+}\left[(n-2)(f(a_i)-f(a_1))^2+\frac{(n-1)^2}{n-2}\left(-\sum_{j=2,j\neq i}^{k}(f(a_j)-f(a_1))+\sum_{j=k+1}^{n}(f(a_1)-f(a_j))\right)^2\right].
\end{align}

Combining \eqref{s+},\eqref{ultimo termino}, \eqref{medio termino} and using that $k\geq2$ we obtain that
\begin{align}\label{s+ conclusion}
    \sum_{i\in S^+}(f(a_i)-m_f)^2\leq&\frac{(n-1)^2+(n-2)}{n^2}\sum_{i\in S^+}(f(a_i)-f(a_1))^2\nonumber\\
    &+\frac{(n-k)}{n^2}(1+\frac{(n-1)^2}{n-2})\sum_{j=k+1}^{n}(f(a_1)-f(a_j))^2\nonumber\\
    &\leq \frac{(n-1)^2+(n-2)}{n^2}\left[\sum_{i\in S^+}(f(a_i)-f(a_1))^2+\sum_{j=k+1}^{n}(f(a_1)-f(a_j))^2\right]. 
\end{align}

Finally, combining \eqref{I+II} and \eqref{s+ conclusion} we conclude that
\begin{align*}
    (\var_2{M_{S_n}f})^2\leq \frac{(n-1)^2+(n-2)}{n^2}\sum_{i=2}^{n}(f(a_i)-f(a_1))^2.
\end{align*}

Moreover, we observe that in order to have an equality in \eqref{s+ conclusion} we need to have $k=2$ (this means that there is only one term larger than $f(a_1)$), in order to have an equality in \eqref{ultimo termino} we need to have $f(a_j)=f(a_{k+1})=f(a_{3})$ for all $j\geq k+1=3$, and in order to have an equality in \eqref{medio termino} we need to have   $(f(a_2)-f(a_1))=(n-1)(f(a_1)-f(a_{3}))$. We verify that if $f(a_1)=x>0$, $f(a_{j})=x-c$ for all $j\geq 3$ and some $c\in(0,x)$, and $f(a_2)=x+c(n-1)$ then we have an extremizer. In fact, in this case we have that $M_{S_n}f(a_2)=x+c(n-1)$ and $M_{S_n}f(a_j)=x+c/n$ for all $j\neq2$. Therefore
$$
\frac{\var_2(M_{S_n}f)}{\var_2f}=\frac{c(n-1-1/n)}{[c^2(n-1)^2+c^2(n-2)]^{1/2}}=\frac{[(n-1)^2+(n-2)]^{1/2}}{n}.
$$

{\it{Case 2: $f(a_1)>m_f$.}} For this case we assume without loss of generality that $f(a_2)\ge f(a_3)\ge \dots f(a_s)> f(a_1)\ge f(a_{s+1})\ge \dots f(a_k)> 2m_f-f(a_1)\ge f(a_{k+1})\ge \dots f(a_n).$ We observe that $M_{S_n}f(a_i)=f(a_i),$ for $i\le s;$ $M_{S_n}f(a_i)=\frac{f(a_1)+f(a_i)}{2}$ for $s<i\le k$ and $M_{S_n}f(a_i)=m_f$ for $i>k.$ We write $f(a_i)-f(a_1)=x_i$ for $i\le s$, $f(a_1)-f(a_i)=y_i$ for $i>s$ and $f(a_1)-m_f=u.$
Then, our goal is to prove 
\begin{align}\label{goal}
\sum_{i=2}^{s}x_i^2+\sum_{i=s+1}^k\left(\frac{y_i}{2}\right)^2+(n-k)u^2\le \left(1-\frac{n+1}{n^2}\right)\left(\sum_{i=2}^{s}x_i^2+\sum_{i=s+1}^{n}y_i^2\right),
\end{align}
since $1-\frac{n+1}{n^2}=\frac{(n-1)^2+(n-2)}{n^2}$.
Assume that $f:V\to\mathbb{R}_{\ge 0}$ is such that $$\frac{Var_{2}(M_{S_n}f)}{Var_{2}(f)}={\bf C}_{S_n,p}.$$ We prove some properties about $f$ following the ideas of Propositions \ref{twovaluescomplete} and \ref{threevaluesstar}.
First, we observe that $s\ge 2.$ Otherwise we would have that LHS in \eqref{goal} is less or equal than
$$\sum_{i=s+1}^{k}\left(\frac{y_i}{2}\right)^2\le \frac{1}{4}Var_2(f)^2<\left(1-\frac{n+1}{n^2}\right)Var_{2}(f)^2.$$ So $f$ could not be an extremizer in that case. \\

{\it{Step 1: $s=2$.}}
We consider $\widetilde{f}:V\to \mathbb{R}$ defined by $\widetilde{f}(a_2)=\sum_{i=2}^{s}f(a_i)-(s-2)f(a_1),$ $\widetilde{f}(a_i)=\widetilde{f}(a_1)$ for $i=3,\dots,s$ and $\widetilde{f}=f$ elsewhere. Clearly $m_{\widetilde{f}}=m_{f}$ then, defining $\widetilde{x}_i$ and $\widetilde{y}_i$ and $\widetilde{u}$ analogously to $x_i$, $y_i$ and $u$, since ${\bf C}_{S_n,2}^2<1$ we observe that
\begin{align}\label{computation}
\begin{split}
0&={\bf C}_{S_n,2}^2\left(\sum_{i=2}^{s}x_i^2+\sum_{i=s+1}^{n}y_i^2\right)-\sum_{i=2}^{s}x_i^2-\sum_{i=s+1}^k\left(\frac{y_i}{2}\right)^2-(n-k)u^2\\
&=({\bf C}_{S_n,2}^2-1)\sum_{i=2}^{s}x_i^2+({\bf C}_{S_n,2}^2-\frac{1}{4})\left(\sum_{i=s+1}^k y_i^2\right)+({\bf C}_{S_n,2}^2)\left(\sum_{i=k+1}^n y_i^2\right)-(n-k)u^2\\
&\ge \left({\bf C}_{S_n,2}^2-1\right)\left(\sum_{i=2}^{s}x_i\right)^2+\left({\bf C}_{S_n,2}^2-\frac{1}{4}\right)\left(\sum_{i=s+1}^k y_i^2\right)+({\bf C}_{S_n,2}^2)\left(\sum_{i=k+1}^n y_i^2\right)-(n-k)u^2\\
&=({\bf C}_{S_n,2}^2-1)\sum_{i=2}^s\widetilde{x_i}^2+\left({\bf C}_{S_n,2}^2-\frac{1}{4}\right)\left(\sum_{i=s+1}^k \widetilde{y}_i^2\right)+({\bf C}_{S_n,2}^2)\left(\sum_{i=k+1}^n \widetilde{y}_i^2\right)-(n-k)\widetilde{u}^2.
\end{split}
\end{align}
Therefore $$\sum_{i=2}^{s}\widetilde{x}_i^2+\sum_{i=s+1}^k \left(\frac{\widetilde{y}_i}{2}\right)^2+(n-k)\widetilde{u}^2\ge {\bf C}_{S_{n},2}^2\left(\sum_{i=2}^s \widetilde{x}_i^2+\sum_{i=s+1}^{n}\widetilde{y}_i^2\right).$$ This implies that $$\frac{Var_2(M_{S_n}\widetilde f)}{Var_2(\widetilde f)}\geq {\bf C}_{S_n,2},$$ thus \eqref{computation} has to be an equality. Then $\sum_{i=2}^sx^2_i=(\sum_{i=2}^{s}x_i)^2$, therefore, there exists at most one $j\in\{2,\dots,s\}$ such that $x_j\neq0$. Since we have that $x_j>0$ for all $j\in\{2,\dots,s\}$ we conclude that $s=2$.\\


{\it{Step 2: $f(a_j)=f(a_3)$ for all $j\in\{3,4,\dots,k\}$}}.
We define the function $\widetilde{f}:V\to \mathbb{R}_{\ge 0}$ as follows: $\widetilde{f}(a_i)=\frac{\sum_{j=3}^{k}f(a_j)}{k-2}$ for every $i\in \{3,\dots k\}$ and $\widetilde{f}=f$ elsewhere. We define $\widetilde{x}_i,$ $\widetilde{y}_i$ and $\widetilde{u}$ analogously to $x_i,$ $y_i$ and $u,$ respectively.
We observe that $\sum_{i=3}^k \widetilde{y}_i=\sum_{i=3}^k y_i,$ and by H\"older's inequality we have $\sum_{i=3}^k \widetilde{y}_i^2\le \sum_{i=3}^k y_i^2$. So, similarly as in \eqref{computation}, since ${\bf C}_{S_n,2}^2>\frac{1}{4},$ we conclude that $\frac{Var_{2}M_{S_n}\widetilde{f}}{Var_{2}(\widetilde{f})}\ge {\bf C}_{S_n,2}.$ Thus $\widetilde f$ is also an extremizer, and 
$\sum_{i=3}^k \widetilde{y}_i^2= \sum_{i=3}^k y_i^2$. This implies that $\widetilde{f}=f.$\\ 

{\it{Step 3: $f(a_j)=f(a_{k+1})$ for all $j\in\{k+1,k+2,\dots,n\}$.}}
Now we define $\widetilde{f}:V\to \mathbb{R}$ as follows: $\widetilde{f}(a_i)=\frac{\sum_{j=k+1}^n f(a_j)}{n-k}$ for every $i\ge k+1,$ and $f=\widetilde{f}$ elsewhere. Then, we have that $$\sum_{i=k+1}^n \widetilde{y}_i=\sum_{i=k+1}^n y_i$$ and $$\sum_{i=k+1}^n\widetilde{y}_i^2\le  \sum_{i=k+1}^n y_i^2.$$ So, by a computation similar than \eqref{computation} we have that $$\frac{Var_{2}M_{S_n}\widetilde{f}}{Var_{2}(\widetilde{f})}\ge {\bf C}_{S_n,2}.$$ Then $\widetilde{f}=f.$ 

{\it{Step 4: Conclusion.}}
So, by now we conclude that $f$ takes at most $4$ values. In fact, we know that $y_i=y_{3}$ for $i\le k$ and $y_i=y_{k+1}$ for $i\ge k+1.$ In the following we conclude that $y_3=y_{k+1}.$ We start observing that if $2m_f-f(a_1)=f(a_j)$ for all $j\in\{k+1,\dots,n\}$ then we can conclude as in the Step 2. Moreover, since $f(a_3)\geq f(a_{k+1})$ we have that $y_3\leq y_{k+1}$.

Let us assume that $y_3<y_{k+1}.$ and there exists $i\in\{k+1,\dots,n\}$ such that $2m_f-f(a_1)>f(a_i)$, We consider now $\widetilde{f}$ defined as follows, $\widetilde{f}(a_k)=f(a_k)-\varepsilon,$ $\widetilde{f}(a_{i})=f(a_{i})+\varepsilon,$ and $\widetilde{f}=f$ elsewhere, where $\varepsilon$ is small enough such that $f(a_k)-\varepsilon>2m_f-f(a_1)>f(a_{i})+\varepsilon$. We observe that $$\frac{Var_{2}M_{S_n}f}{Var_{2}f}<\frac{Var_{2}M_{S_n}\widetilde{f}}{Var_{2}\widetilde{f}}.$$ In fact,
$$Var_{2}M_{S_n}{f}=\sum_{j=2}^sx_{j}^2+\sum_{j=s+1}^{k}\frac{y_j^2}{4}+(n-k)u^2< \sum_{j=2}^sx_{j}^2+\sum_{j=s+1}^{k-1}\frac{y_j^2}{4}+\frac{(y_k+\varepsilon)^2}{4}+(n-k)u^2=Var_{2}M_{S_n}\widetilde{f}$$
and 
$$Var_{2}\widetilde{f}=\sum_{j=2}^{s}x_j^2+\sum_{j=s+1}^{k-1}y_j^2+(y_k+\varepsilon)^2+(y_{i}-\varepsilon)^2+\sum_{j=k+1,j\neq i}^ny_j^2<\sum_{j=2}^{s}x_j^2+\sum_{j=s+1}^ny_j^2=Var_{2}f$$ 
for $\varepsilon$ small enough, since $(y_k+\varepsilon)^2+(y_{k+1}-\varepsilon)^2=y_{k}^2+y_{k+1}^2+2\varepsilon(y_{k}-y_{k+1})+2\varepsilon^2<y_{k}^2+y_{k+1}^2$ given that $y_{k}-y_{k+1}+\varepsilon=y_3-y_{k+1}+\varepsilon<0$ for $\varepsilon$ small enough.  Therefore $\frac{Var_{2}M_{S_n}f}{Var_{2}f}<\frac{Var_{2}M_{S_n}\widetilde{f}}{Var_{2}\widetilde{f}},$ contradicting the fact that $f$ is an extremizer. Then, $y_3=y_{k+1}$ or equivalently $f(a_3)=f(a_{k+1})$, therefore $f$ only takes three values. Now we have only two subcases left to analyse: 
\begin{itemize}
    \item {\it{Subcase 1: $f(a_1)+f(a_n)\ge 2m_{f}.$}} In this case $\frac{f(a_1)+f(a_n)}{2}=M_{S_n}f(a_i)$ for $i=3,\dots,n$ and $y_3=y_i$ for $i=3,\dots,n.$ Also, we observe that $y_3(n-2)=x_2+nu$ and $u\ge \frac{y_3}{2}.$ Then we need to prove that $$x_2^2+(n-2)\frac{y_3^2}{4}\le \left(1-\frac{n+1}{n^2}\right)(x_2^2+(n-2)y_3^2),$$
    or, equivalently,\begin{align*}
    \frac{n+1}{n^2}x_2^2\le (n-2)\left(3/4-\frac{n+1}{n^2}\right)y_3^2.
    \end{align*} 
    Since $y_3(n-2)=x_2+nu\ge x_2+\frac{n}{2}y_3$ we have $y_3\left(\frac{n}{2}-2\right)\ge x_2,$ therefore is enough to prove  $$\left(\frac{n}{2}-2\right)^2\frac{n+1}{n^2}\le (n-2)\left(3/4-\frac{n+1}{n^2}\right),$$ and that can be established for $n\ge 3.$
    \item {\it{Subcase 2: $f(a_1)+f(a_2)\le 2m_{f}.$}} In this case $M_{S_n}f(a_i)=m_f$ for $i=3,\dots,n$. Also, we observe that $u\le \frac{y_3}{2}$. Thus we need to prove that 
    $$x_2^2+(n-2)u^2\le \left(1-\frac{n+1}{n^2}\right)(x_2^2+(n-2)y_3^2),$$
    or equivalently $$x_2^2\left(\frac{n+1}{n^2}\right)+(n-2)u^2\le (n-2)\left(1-\frac{n+1}{n^2}\right)y_3^2.$$ Indeed, since $y_3(n-2)=x_2+nu$ we have $$y_3^2\ge \frac{x_2^2}{(n-2)^2}+\frac{n^2}{(n-2)^2}u^2.$$ Then is enough to prove $\frac{n+1}{n^2}\le \frac{\left(1-\frac{n+1}{n^2}\right)}{n-2}$ and $(n-2)^2\le n^2\left(1-\frac{n+1}{n^2}\right).$ Since both hold for $n\ge 3,$ we conclude.
\end{itemize}
\end{proof}

\subsection{The $p$-variation of $M_{G}$.}  
For a finite connected graph $G=(V,E)$ with vertices $V=\{a_1,a_2,\dots,a_n\}$ we define $d(G)=:\max\{d_{G}(a_i,a_j);a_i,a_j\in V\}$ and
$\Omega_{G}:=\{a_i\in V; \exists a_j \in V\ \text{such that}\ \ d(G)=d(a_i,a_j)\}$. For all $H\subset G$ we choose a minimum degree element of $H$ and we denote this by $a_H$.

\begin{proposition}
Let $G$ be a finite connected graph with $n$ vertices, assume that $\deg(a_{\Omega_G})=k$ and there exists a vertex $x\in V$ such that $d(x,a_{\Omega_G})\geq d(x,y)$ for all $y\in V$ and there are $k$ disjoint paths from $a_{\Omega_G}$ to $x$. Then
$$
{\bf C}_{G,p}\geq 1-\frac{1}{n}
$$
for all $p\in (0,1]$.
\end{proposition}
\begin{proof}
This result follows observing that under these hypothesis we have that
$$
{\bf C_{G,p}}\geq \frac{\var_pM_{G}\delta_{a_{\Omega_G}}}{\var_p\delta_{a_{\Omega_G}}}=\frac{\var_pM_{G}\delta_{a_{\Omega_G}}}{k^{1/p}}\geq 1-\frac{1}{n}.
$$
\end{proof}

In particular this results hold for trees (in that case $k=1$), cycles (in that case $k=2$), hypercubes $Q_n$ (with $2^n$ vertices, in that case $k=n$), whenever $k=1$, etc. 

\begin{remark}
For all $p\in(0,1)$ we have that
${\bf C}_{L_n,p}>1-\frac{1}{n},$ here $L_n$ is the line graph.  This also happens in many other situations. Moreover, it was proved by the authors in a previous work \cite{GRM} that ${\bf C}_{S_n}=1-\frac{1}{n}$ for all $p\in[1/2,1]$ and similarly ${\bf C}_{K_n}=1-\frac{1}{n}$ for all $p\geq\frac{\ln 4}{\ln 6}$
\end{remark}

Let $\Gamma_n$ the family of all connected simple finite graphs with n vertices. Our previous proposition motivates the following question.\\

{\it{Question A:}} Let $p>0$. What are the values 
$$
c_{n,p}=\inf_{G\in \Gamma_n}{\bf C}_{G,p}\ \ \text{and}\ \  C_{n,p}=\sup_{G\in\Gamma_n}{\bf C}_{G,p}?
$$
Moreover, what are the extremizers? i.e what are the graphs $G\in \Gamma_n$ for which ${\bf C}_{G,p}=C_{n,p}$ or ${\bf C}_{G,p}=c_{n,p}$?


\section{Discrete Hardy-Littlewood maximal operator}

In this section we write $M:=M_{\mathbb{Z}}$.
The following result was proved by the second author for $p=1$ in \cite{Madrid2016}. This proof follows a similar strategy, we include some details for completeness.

\begin{theorem}\label{lim d=1 C=2}
Let $p\in(0,1]$ and $f:\Z\to\R$ be a function in $\ell^{p}(\Z).$ Then
\begin{equation}\label{main theo cent d=1}
\var_p Mf\leq \left(2\sum_{k=0}^{\infty}\frac{2^{p}}{(2k+1)^p(2k+3)^p}\right)^{\frac{1}{p}}\|f\|_{\ell^{p}(\Z)}=:{\bf C}_p\|f\|_p,
\end{equation}
and the constant ${\bf C}_p$ is the best possible. Moreover, the equality for $p\in (\frac{1}{2},1]$ is attained if and only if $f$ is a delta function.
\end{theorem} 

\begin{proof}[Proof of Theorem \ref{lim d=1 C=2}]
 We can assume without loss of generality that $f\geq0$. We observe that for all $n\in\Z$ there exists $r_{n}\in\Z$ such that $Mf(n)=A_{r_{n}}f(n)$ (this follows from the fact that $f\in l^p(\Z)$), then we consider the sets
$$
X^{-}=\{n\in\Z; Mf(n)>Mf(n+1)\}\ \  \text{and}\ \ X^{+}=\{n\in\Z; Mf(n+1)>Mf(n)\}.
$$
\noindent
Then
\begin{eqnarray}\label{suma}
(\var_p Mf)^p&=&\sum_{n\in\Z}|Mf(n)-Mf(n+1)|^p\nonumber\\
&\leq&\sum_{n\in X^{-}}(A_{r_{n}}f(n)-A_{r_{n}+1}f(n+1))^P+\sum_{n\in X^{+}}(A_{r_{n+1}}f(n+1)-A_{r_{n+1}+1}f(n))^p. 
\end{eqnarray}


Observe that for all $n\in X^-$ we have that
\begin{equation}\label{term 1}
(A_{r_n}f(n)-A_{r_n+1}f(n+1))^p\leq \left|\frac{2}{(2r_n+1)(2r_n+3)}\sum_{k=n-r_n}^{n+r_n}f(k)\right|^p\leq \frac{2^p}{(2r_n+1)^p(2r_n+3)^p}\sum_{k=n-r_n}^{n+r_n}f(k)^p.
\end{equation}
Then, for any $m\in\Z$ fixed, we find
the maximal contribution of $f(m)^p$ to the right hand side of \eqref{term 1}.

\textbf{Case 1:} If $n\geq m$.\\
Since $n\in X^{-}$. In this case we have that the contribution of $f(m)$ to the right hand side of \eqref{term 1} is $0$ (if $m<n-r_{n}$) or $\frac{2^p}{(2r_{n}+1)^p(2r_{n}+3)^p}$ (if $n-r_{n}\leq m$). Thus, the contribution of $f(m)$ to $(A_{r_{n}}f(n)-A_{r_{n}+1}f(n+1))^p$ is at most
\begin{eqnarray*}
\frac{2^p}{(2r_{n}+1)^p(2r_{n}+3)^p}
\leq \frac{2^p}{(2(n-m)+1)^p(2(n-m)+3)^p}.
\end{eqnarray*}
Here the equality happen if and only if $r_{n}=n-m$.\\


\textbf{Case 2:} If $n<m$.\\
Since $n\in X^{-}$. In this case we have that the contribution of $f(m)^p$ to the right hand side of \eqref{term 1} is $0$ (if $m>n+r_{n}$) or $\frac{2^p}{(2r_{n}+1)^p(2r_{n}+3)^p}$ (if $n+r_{n}\geq m$). Thus, the contribution of $f(m)^p$ to the right hand side of \eqref{term 1} is at most
\begin{eqnarray*}
\frac{2^p}{(2r_{n}+1)^p(2r_{n}+3)^p}
&\leq& \frac{2^p}{(2(m-n)+1)^p(2(m-n)+3)^p}< \frac{2^p}{(2(m-n-1)+1)^p(2(m-n-1)+3)^p}.
\end{eqnarray*}
\\

A similar analysis can be done with the second term of \eqref{suma}, in fact, for a fixed $n\in X^+$ we start observing that
\begin{align*}\label{term 2}
(Mf(n+1)-Mf(n))^p&\leq
(A_{r_{n+1}}f(n+1)-A_{r_{n+1}}f(n))^p\\
&\leq \left|\frac{2}{(2r_{n+1}+1)(2r_{n+1}+3)}\sum_{k=n+1-r_{n+1}}^{n+1+r_{n+1}}f(k)\right|^p\\
&\leq \frac{2^p}{(2r_{n+1}+1)^p(2r_{n+1}+3)^p}\sum_{k=n+1-r_{n+1}}^{n+1+r_{n+1}}f(k)^p.
\end{align*}

Then, if $n\geq m$, the contribution of $f(m)^p$ to the previous expression is strictly smaller than 
$$
\frac{2^p}{(2(n-m)+1)^p(2(n-m)+3)^p}.
$$
Moreover, if $n< m$, the contribution of $f(m)^p$ is smaller than or equal to 
$$
\frac{2^p}{(2(m-n-1)+1)^p(2(m-n-1)+3)^p}.
$$

Therefore, from \eqref{suma} we conclude that
\begin{align*}
(\var_p Mf)^p&\leq \left[\sum_{ m=-\infty}^{n}\frac{2^p}{(2(n-m)+1)^p(2(n-m)+3)^p}+\sum_{m=n+1}^{+\infty}\frac{2^p}{(2(m-n-1)+1)^p(2(m-n-1)+3)^p}\right]\|f\|^p_p\\
&= 2\sum_{k=0}^{\infty}\frac{2^p}{(2k+1)^p(2k+3)^p}\|f\|^p_p.
\end{align*}
We can easily see that if $f$ is a delta function then the previous inequality becomes an equality.
On the other hand, for a function $f:\Z\to\R$ such that $(\var_p Mf)^p=2\sum_{k=0}^{\infty}\frac{2^p}{(2k+1)^p(2k+3)^p}\|f\|^p_p$ and $f\geq0$, we consider the set $P:=\{s\in\Z; f(s)\neq 0\}$, thus
$$
(\var_p Mf)^p=2\left(\sum_{k=0}^{\infty}\frac{2^p}{(2k+1)^p(2k+3)^p}\right)\sum_{t\in P}f(t)^p.
$$
Then, given $s_{1}\in P$, 
by the previous analysis we note that for all $n\geq s_{1}$ we must have that $n\in X^{-}$ and $r_{n}=n-s_{1}$. If we take $s_{2}\in P$ the same has to be true, this implies that $s_{1}=s_{2}$, therefore $P=\{s_{1}\}$ which means that $f$ is a delta function. 

\end{proof}

Observe that for all $p\in(0,1]$
the following inequality holds
$$\var_pf \leq 2^{1/p}\|f\|_p$$
for any function $f:\Z\to\R$. This follows from the fact that
$|f(n)-f(n+1)|^p\leq |f(n)|^p+|f(n+1)|^p$ for all $n\in\Z$. Motivated by this trivial bound and our Theorem \ref{lim d=1 C=2} we pose the following question

\begin{conjecture}\label{quest: in z p<1}
Let $p\in(1/2,1]$ and $f:\Z\to\R$ be a function in $\ell^{p}(\Z).$ Then
\begin{equation}\label{conj ineq for p<1 in z}
\var_p Mf\leq \left(\sum_{k=0}^{\infty}\frac{2^{p}}{(2k+1)^p(2k+3)^p}\right)^{\frac{1}{p}}\var_pf.
\end{equation}
\end{conjecture}
In general, it would be interesting to answer the following question:\\
{\it{Question B:}} Let $p\in(0,\infty]$. What is the smallest constant $C_p$ such that
$$
\|(Mf)'\|_p=\var_pMf\leq C_p\var f=\|f'\|_p,
$$
for all $f:\Z\to\R$.\\

We note that for $p=\infty$ we have that $C_{\infty}=1$. The upper bound $C_{\infty}\leq 1$ trivially holds, on the other hand to see that the lower bound $C_{\infty}\geq 1$ holds it is enough to consider the function $f:\Z\to\R$ defined by $f(n)=\max\{10-|n|,0\}$. Moreover, observe that for $p\leq 1/2$ the right hand side of \eqref{conj ineq for p<1 in z} is $+\infty$ for any no constant function, so the inequality \eqref{conj ineq for p<1 in z} trivially holds in that case. However, this is highly not trivial for $p\in (1/2,1]$.
If true, this results would be stronger than our Theorem \ref{lim d=1 C=2}. For $p>1$ even the analogous result to our Theorem \ref{lim d=1 C=2} remains open.

For $p=1$, it was proved by Kurka \cite{Kurka2010} that 
$$
\var Mf\leq 240,004\var f,
$$
where $\var_1=\var$.
In this case Conjecture \ref{quest: in z p<1} simplifies to 
$$
\var Mf\leq \var f,
$$
which is believe to be true. This problem, innocent at first sight, has shown to be very challenging, and remains open. Any significant progress  reducing the constants towards the conjectural bounds would be very interesting. This problem has also been considered for the uncentered Hardy-Littlewood maximal operator $\widetilde M$, in \cite{BCHP2012} Bober, Carneiro, Hughes and Pierce proved that for all $f:\Z\to\R$ the following optimal inequality holds
$$
\var \widetilde Mf\leq \var f.
$$
Then, also would be interesting to answer the following question:\\
{\it{Question C:}}  Let $p\in(0,\infty]$. What is the smallest constant $\widetilde C_p$ such that
$$
\|(\widetilde Mf)'\|_p=\var_p\widetilde Mf\leq\widetilde C_p\var f=\widetilde C_p\|f'\|_p,
$$
for all $f:\Z\to\R$. Our next theorem gives an answer to this question for $p=\infty$. An auxiliary tool is the following lemma. 

\begin{lemma}\label{lemma infty}
Let $f:\Z \to \R^+$ be a function such that $\|f'\|_{\infty}<\infty$ and $\wt{M}_{}f \not\equiv \infty$. 
Then, we have $\wt{M}_{}f(n) <\infty$ for all $n \in \Z$.
\end{lemma}
\begin{proof}
Assume that there is $n \in \Z$ such that $\wt{M}_{}f(n) =\infty$, then, there exists a sequence $\{r_j, s_j\}$ in $\Z^+ \times \Z^+$, with $r_j + s_j \to \infty$ such that $A_{r_j,s_j}f(n) \to \infty$ as $j \to \infty$. For any $m \in \Z$, defining $C = \|f'\|_{\infty}$ we have
\begin{equation}\label{Sec4_lem6_eq2}
A_{r_j,s_j}f(m) \geq A_{r_j,s_j}f(n) - C|m-n|\,,
\end{equation}
therefore $\wt{M}_{\beta}f(m) =\infty$, a contradiction.
\end{proof}

\begin{theorem}\label{p=infty}
For all $f:\mathbb{Z}\to\mathbb{R}$ such that $\wt{M}_{}f \not\equiv \infty$, we have that
$$
\|(\widetilde Mf)'\|_{\infty}\leq \frac{1}{2}\|f'\|_{\infty}.
$$
Moreover, the equality is attained if $f$ is a delta function. 
\end{theorem}
\begin{remark}
This theorem is a discrete analogue of the main result obtained in \cite{Aldaz2010} on the continuous setting, in that case the optimal constant is $2^{1/2}-1$. We use an elementary combinatorial argument to establish our result, this technique is completely independent of those in \cite{Aldaz2010}. 
\end{remark}

\begin{proof}

We assume without loss of generality that $f$ is nonnegative. Let $n\in\Z$, by Lemma \ref{lemma infty} we have that $\widetilde Mf(n)<\infty$, then, for all $\varepsilon>0$ there are $r_{n,\varepsilon},s_{n,\varepsilon}\geq0$ such that
\begin{equation}\label{n epsilon}
\widetilde Mf(n)<\frac{1}{r_{n,\epsilon}+s_{n,\varepsilon}+1}\sum_{k=-s_{n,\varepsilon}}^{r_{n,\varepsilon}}f(n+k)+\varepsilon.
\end{equation}
We analyze two cases, the argument works similarly for both situations.
{\it{Case 1:}} $ (\widetilde Mf)'(n)>0$. 
In this case we star observing that $r_{n,\varepsilon}=0$ for all sufficiently small $\varepsilon$ (otherwise, from \eqref{n epsilon} we would obtain $\widetilde Mf(n)\leq \widetilde Mf(n+1)$). Then, for all sufficiently small $\varepsilon$
we have that
\begin{align*}
    \widetilde Mf(n)-\widetilde Mf(n+1)&\leq 
    \frac{1}{s_{n,\varepsilon}+1}\sum_{k=-s_{n,\varepsilon}}^{0}f(n+k)+\varepsilon-\frac{1}{s_{n,\varepsilon}+2}\sum_{k=-s_{n,\varepsilon}-1}^{0}f(n+1+k)\\
    &\leq \left(\frac{1}{s_{n,\varepsilon}+1}-\frac{1}{s_{n,\varepsilon}+2}\right)\sum_{k=-s_{n,\varepsilon}}^{0}f(n+k)-\frac{1}{s_{n,\varepsilon}+2}f(n+1)+\varepsilon\\
    &=\frac{1}{(s_{n,\varepsilon}+2)(s_{n,\varepsilon}+1)}\sum_{k=-s_{n,\varepsilon}}^{0}(f(n+k)-f(n+1))+\varepsilon\\
    &\leq \frac{1}{(s_{n,\varepsilon}+2)(s_{n,\varepsilon}+1)}\sum_{k=1}^{s_{n,\varepsilon}+1}k\|f'\|_{\infty}+\varepsilon\\
    &= \frac{1}{(s_{n,\varepsilon}+2)(s_{n,\varepsilon}+1)}\frac{(s_{n,\varepsilon}+1)(s_{n,\varepsilon}+2)}{2}\|f'\|_{\infty}+\varepsilon\\
    &= \frac{1}{2}\|f'\|_{\infty}+\varepsilon.
\end{align*}
Since this holds for any arbitrary $\varepsilon$, sending $\varepsilon$ to $0$ we conclude that
$$
\widetilde Mf(n)-\widetilde Mf(n+1)\leq \frac{1}{2}\|f'\|_{\infty}.
$$
{\it{Case 2:}} $(\widetilde Mf)'(n)<0$.  This case follows anlogously. Since these are the only two possible cases
the result follows.
\end{proof}

\bibliography{Reference}
\bibliographystyle{amsplain}

\end{document}